\documentclass[runningheads]{llncs}

\usepackage{epstopdf}
\usepackage[caption=false]{subfig}

\usepackage{amsmath, amssymb}
\usepackage{mathtools}
\usepackage{graphicx}
\mathtoolsset{showonlyrefs}
\usepackage{xcolor}

\usepackage{hyperref}
\hypersetup{
    colorlinks=true,
    citecolor=blue,
    linkcolor=red, 
    filecolor=magenta,      
    urlcolor=cyan,
}

\usepackage{wrapfig}

\usepackage{xcolor}
\usepackage{subfig}
\usepackage{hyperref}
\usepackage{textcomp}
\usepackage{nicefrac}

\usepackage{algorithm}
\usepackage{algpseudocode}
 \usepackage{algorithmicx}
\algnewcommand{\LeftComment}[1]{\Statex \(\triangleright\) #1}

\usepackage{xcolor}
\usepackage{bbm}

\usepackage{graphicx}
\begin{document}
\title{Network utility maximization by updating individual transmission rates\thanks{The research is supported by the Ministry of Science and Higher Education of the Russian Federation (Goszadaniye) №075-00337-20-03, project No. 0714-2020-0005.}}
%

\author{Dmitry Pasechnyuk\orcidID{0000-0002-1208-1659}}
\authorrunning{D. Pasechnyuk}
%
\institute{Moscow Institute of Physics and Technology, Moscow, Russia}
\maketitle              
\begin{abstract}
This paper discusses the problem of maximizing the total data transmission utility of the computer network. The total utility is defined as the sum of the individual (corresponding to each node in the network) utilities that are concave functions of the data transmission rate. For the case of non-strongly concave utilities, we propose an approach based on the use of a fast gradient method to optimize a dually smoothed objective function. As an alternative approach, we introduce stochastic oracles for the problem under consideration and interpret them as the messages on the state of some individual node to use randomized switching mirror descent to solve the problem above. We propose interpretations of both described approaches allowing the effective implementation of the protocols of their operation in the real-life computer networks environment, taking into account the distributed information storage and the restricted communication capabilities. The numerical experiments were carried out to compare the proposed approaches on sythetic examples of network architectures.

\end{abstract}

\section{Introduction}

Management of the operation of computer communication networks and support of their efficiency, in view of their widespread and massive use, are the crucial tasks today. In particular, the proper management of data transmission in the network should ensure that there are no overloaded, and therefore slow, connections between computers. The network protocol in this setting is not a determining factor; you can fix it and imagine, for example, TCP Internet traffic. At the same time, the task of data flow control is then passed on to the individual computing agents, i.e. network participants. The natural way to control overloading of connections for them is the ability to change their own data transmission rates.

Of course, there are many algorithms and protocols for finding the optimal data transmission rates \cite{arrow1958decentralization,campbell1987resource,friedman1995complexity,kakhbod2013resource,rokhlin2021resource}. One of the currently developing approaches proposes to introduce a function of the total utility of the network and consider the optimization problem of maximizing utility in relation to variable data transmission rates \cite{kelly1998rate}. Thus, some intermediary or the network as a whole generate a sequence of values of transmission rates, tending to the minimum of the introduced potential. In such a setting, in order to provide the most resource-efficient procedure for finding the optimal rates, one can use the apparatus and methods developed by the modern theory of convex optimization.

This paper discusses some methods for efficiently optimizing data transmission rates. The first of them follows the idea of setting prices for data transmission over connections \cite{ivanova2020composite,nesterov2018dual}, so that transmission rates are chosen by each computer for reasons of maximum utility minus cost. The second uses a switching scheme to change the speeds alternately to more profitable and less loading ones.

However, for the practical application of the schemes proposed in various papers, it is important to represent their implementation in the real-life architecture of a computer network. For both approaches, we describe operating protocols that allow efficiently making updating and storing values distributed, partially parallelized and encapsulated, without losing efficiency. We analyze each method for the convergence rate and indicate, in addition, the dependence of the efficiency on the characteristics of the network and the problem.

This paper is organized as follows. Section \ref{section_problem} describes the task of utility maximization in computer networks and formulates the main optimization problem. Further, Section \ref{section_smoothing} using the dual smoothing technique introduce the dual optimization problem and analyses the properties of it. Section \ref{section_fgm} describes the primal-dual fast gradient method and analyses its convergence rate for the considered dual problem. We also describe the possibilities of the distributed implementation for this algorithm. Section \ref{section_mirror} consider a different type of algorithm, i.e. randomized version of switching mirror descent. We provide the corresponding convergence theorem for the considered problem, and give a description of the possible method operation protocol in a real-life computer network environment. Finally, Section \ref{section_num} describes experiments on the application of the considered optimization methods to some synthetic computer networks.

\section{Problem statement} \label{section_problem}

By a computer network we mean a structure consisting of a set of vertices of size $n \in \mathbb{N}$, and a set of connections of size $m \in \mathbb{N}$. Each connection in such a model, unlike the edges in a graph, can have a relationship with more than two vertices. The structure of the relationship between connections and vertices is expressed by the matrix $C \in \{0, 1\}^{m \times n}$ according to the following rule: $C_{j i} = 1 \Leftrightarrow j$-th connection is in relation to $i$-th vertex. It turns out to be natural to distinguish separately such statements in which the matrix $C$ is sparse, i.e. the number of its nonzero elements $nnz(C) \ll m \cdot n$.

Data transmission in the network will be characterized from the point of view of each vertex $i$ by the value $x_i \in \mathbb{R}_+$ of its data transmission rate. Quantity $x_i$ determines how much the vertex $i$ loads each of the connections $j$ that are in the relation with it. Below we use the notation $x = (x_1, ..., x_n)^\top \in \mathbb{R}_+^n$. In turn, each of the connections $j$ is characterized by some upper bound on the total load, equal to the throughput $b_j \in \mathbb{R}_+$. Together they form a vector $b = (b_1, ..., b_m)^\top \in \mathbb{R}^m$. One can see that the constraints on the total load of each of the connections is met if and only if the inequality $C x \leq b$ holds.

The utility achieved by the $i$-th vertex is characterized by the utility function $u_i(x_i): \mathbb{R}_+ \rightarrow \mathbb{R}$, that depends on the data transmission rate of one $i$-th vertex. In reasonable formulations of the problem, this function is chosen concave in $x_i$. By the total utility taken by the entire network, we mean the function $U(x) = \sum_{i=1}^n u_i(x_i)$.

The practical task is to find such values of the data transmission rates for each of the vertices, for which the greatest total utility is achieved. Taking into account the additional details specified above, we can now formulate the main optimization problem:
\begin{equation}\label{primal}
    \max_{x: C x \leq b} \left\{U(x) = \sum_{i=1}^n u_i(x_i)\right\},
\end{equation}
where $x \in \mathbb{R}_+^n$, $C \in \mathbb{R}^{m \times n}$, $b \in \mathbb{R}^m$, and $u_i$  is differentiable and concave for all $i$.

\section{Dual smoothing} \label{section_smoothing}

Let us move from the formulated main optimization problem with linear constraints to the corresponding dual problem:
\begin{equation}
    \min_{\lambda} \left\{ \varphi(\lambda) = \max_{x} \left\{U(x) + \langle \lambda, b - C x \rangle\right\} \right\},
\end{equation}
where $\lambda \in \mathbb{R}^m_+$. Note that the resulting dual factors can be easily interpreted in terms of our subject area. Indeed, by $\lambda = (\lambda_1, ..., \lambda_m)^\top$ we mean the vector of prices for using a unit of bandwidth for each of the connections. Then, rewriting the definition above in the following form:
\[
    \varphi(\lambda) = \langle \lambda, b \rangle + \sum_{i=1}^n \left[u_i(x_i(\lambda)) - \langle \lambda, C^\top_i x_i(\lambda) \rangle \right],
\]
using the notation
\begin{equation}\label{greedy}
    x_i(\lambda) = \arg \max_{x} \left\{ u_i(x_i) - \langle \lambda, C^\top_i x_i \rangle \right\},
\end{equation}
we will obtain a natural interpretation for the dual problem: it thus consists in finding such values of the prices of connections that, provided that each of the vertices is chosen the most profitable for it (in term of utility $u_i(x_i)$ and costs $\langle \lambda, C^\top_i x_i(\lambda) \rangle$) values of data transmission performance, the term associated with violation of the constraints will be the smallest.

We now note the following fact: in the general case, the function $U$ is not strongly concave. At the same time, it is easy to show that in this case the function $\varphi$ is not necessarily Lipschitz smooth, and this would deprive us of the opportunity to use many of the methods for its optimization, including the fast gradient method. To deal with it, one can use the dual smoothing technique, i.e. instead of the problem \eqref{primal} consider the optimization problem for the regularized function $U_\mu$ with the regularization coefficient $\mu \sim \varepsilon / \| x_* - x_0 \|_p^2$, where $\varepsilon$ is the required accuracy of solving the primal problem. With this choice of $\mu$, it is known that if
\[
    \max_x U_\mu(x) - U_\mu(x_*) \leq \frac{\varepsilon}{2},
\]
is satisfied for some $x_*$, then $x_*$ is also a $\varepsilon$-solution of the problem \eqref{primal}.

Let us now move on to a more formal level, namely, we equip the space $\mathbb{R}^n$ with the norm $\|\cdot\|_p$, and the space $\mathbb{R}^m$ with the norm $\|\cdot\|_q$ for some $p, q \in [1, 2]$. We denote the regularized function as $U_\mu$ and its dual as follows:

\begin{equation}
    U_\mu(x) = U(x) - \frac{\mu}{2} \|x - x_0\|^2_p,\quad \varphi_\mu(\lambda) = \max_{x} \left\{U(x) + \langle \lambda, b - C x \rangle - \frac{\mu}{2} \|x - x_0\|^2_p\right\}
\end{equation}

The following lemma describes the properties of the resulting function $\varphi_\mu$:

\begin{lemma}\label{lipschitz}
The function $\varphi_\mu $ has Lipschitz continuous gradient, i.e. $\forall \lambda_1, \lambda_2 \in \mathbb{R}^m_+:$
\[
    \|\nabla \varphi_\mu(\lambda_1) - \nabla \varphi_\mu(\lambda_2)\|_q \leq L \|\lambda_1 - \lambda_2\|_q,
\]
for $L = \|C\|^2_{p,q} / \mu$, where
\[
    \|C\|_{p,q} = \max_{\|x\|_p=1, \|\lambda\|_q=1} \langle \lambda, C x \rangle
\]
\end{lemma}
\begin{proof}
Literally coincides with the proof of Theorem 1 from \cite{nesterov2005smooth}.
\end{proof}

Assuming at the same time that $\|x_* - x_0\|^2_p \leq R_p^2$, for an important special case of $p = q = 2$ choosing $\mu = \varepsilon / R_p^2$ we have
\[
    L = \frac{R_p^2 \lambda_{max}(C^\top C)}{\varepsilon},
\]
where $\lambda_{max}(A)$ is the largest eigenvalue of the matrix $A$. The next lemma provides another, somewhat more intuitive in the framework of the considered subject area, bound for the Lipschitz constant of the $\varphi_\mu$ gradient.

\begin{lemma}
When $p = q = 2$, function $\varphi_\mu$ has $L$-Lipschitz continuous gradient with $L = nnz(C) / \mu$
\end{lemma}
\begin{proof}
From the first order optimality conditions for \eqref{greedy}:
\[
    \langle \nabla u_i(x_i(\lambda_1)) - \langle \lambda_1, C_i^\top \rangle - \mu (x_i(\lambda_1) - [x_0]_i), x_i(\lambda_1) - x_i(\lambda_2) \rangle \rangle \geq 0,
\]
\[
    \langle \nabla u_i(x_i(\lambda_2)) - \langle \lambda_2, C_i^\top \rangle - \mu (x_i(\lambda_2) - [x_0]_i), x_i(\lambda_2) - x_i(\lambda_1) \rangle \rangle \geq 0.
\]
Summing up, we have:
\begin{align*}
    &\mu \|x_i(\lambda_1) - x_i(\lambda_2)\|_2^2 \\
    &\leq \langle \nabla u_i(x_i(\lambda_2)) - \nabla u_i(x_i(\lambda_1)) - \mu (x_i(\lambda_1) - x_i(\lambda_2)), x_i(\lambda_1) - x_i(\lambda_2) \rangle \\
    &\leq \langle \langle \lambda_1, C_i^\top \rangle - \langle \lambda_2, C_i^\top \rangle, x_i(\lambda_1) - x_i(\lambda_2) \rangle,
\end{align*}
whence
\begin{align*}
     \|\nabla \varphi_i(\lambda_1) - \nabla \varphi_i(\lambda_2)\|_2 \leq \|C_i^\top\|_2 \cdot \|x_i(\lambda_1) - x_i(\lambda_2)\|_2 \leq \frac{\|C_i^\top\|^2_2}{\mu} \|\lambda_1 - \lambda_2\|_2.
\end{align*}
Summing over components and taking into account $C_{j i} \in \{0, 1\}$, we obtain to the expression:
\[
    \|\nabla \varphi(\lambda_1) - \nabla \varphi(\lambda_2)\|_2 \leq \frac{\sum_{i=1}^n \|C_i^\top\|^2_2}{\mu} \|\lambda_1 - \lambda_2\|_2 = \frac{nnz(C)}{\mu} \|\lambda_1 - \lambda_2\|_2.
\]
\end{proof}

Thus, the properties of the optimized function directly depend on the sparsity of the matrix $C$, in other words, in terms of our domain, the fewer vertices on average are in relation to one connection, the less time-consuming the process of finding a solution to the problem is. Further, the nature of this dependence will be refined in the convergence rate bounds for the methods.

\section{Fast gradient method} \label{section_fgm}
\subsection{Theoretical guarantees}

	\begin{algorithm}[H]
		\caption{Primal-dual Fast Gradient Method}\label{alg_fgm}
		\begin{algorithmic}[1]
			
			\Require $\lambda_0$.

            \State $\alpha_{t} = \frac{t+1}{2}$
			
			\State $A_{-1} = 0$, $A_t = A_{t-1} + \alpha_{t} = \frac{(t+1)(t+2)}{4}$
			
			\State $\tau_{t} = \frac{\alpha_{t+1}}{A_{t+1}}=\frac{2}{t + 3}$

			\For{$t=0,\, 1, \, \ldots, \, N-1$}
    			\State Evaluate $\varphi_\mu(\lambda_t)$, $\nabla \varphi_\mu(\lambda_t)$ 
    			\State $y_t = \left[ 
    			    \lambda_t - \frac{1}{L} \left( b - C x(\lambda_t) \right)
                \right]_+$
                
                \State $z_t = \left[
    				\lambda_0 - \frac{1}{L} \sum_{k=0}^t \alpha_k \left(b - C x(\lambda_k) \right)
    		    \right]_+$
    		    
    			\State $\lambda_{t+1} = \tau_t z_t + (1 - \tau_t) y_t$
            \EndFor
			\State\Return $\lambda_N$, $\hat{x}_N = \frac{1}{A_{N}} \sum_{t=0}^{N} \alpha_t x(\lambda_t)$
		\end{algorithmic}
	\end{algorithm}

To analyze the method, we introduce the notation:
\[
\psi_t(\lambda) = \sum_{k = 0}^t \alpha_k \left [ 
\varphi(\lambda_k) + \langle
			\nabla \varphi (\lambda_k), \lambda - \lambda_k\rangle 
			\right ] + \frac{L}{2} \|\lambda - \lambda_0 \|_q^2.
\]

\begin{lemma}\label{lemma_ineq}
\begin{equation}\label{ineq}
A_N \varphi(y_N) \leq \min_{\lambda}
\psi_N (\lambda) = \psi_N(z_N).
\end{equation}
\end{lemma}

\begin{proof}
Let us prove by induction that \eqref{ineq} holds.
When $t = 0$, the following holds:
\[
\psi_0(z_0) = \min_{\lambda}
\left \{
\alpha_0 \left[ 
\varphi_\mu(\lambda_0) + \langle
			\nabla \varphi_\mu(\lambda_0), \lambda - \lambda_0\rangle 
			\right] +
			\frac{L}{2} \|\lambda - \lambda_0\|_q^2
\right \} \geq \alpha_0 \varphi_\mu(y_0).
\]
Let \eqref{ineq} to be true for $t$: $A_t \varphi(y_t) \leq \psi_t(z_t)$. Let us prove that \eqref{ineq} is true for $t+1$.
Indeed, we have
\begin{align}
\nonumber &\psi_{t+1} (z_{t+1}) = \min_{\lambda}
\left\{
\psi_t(\lambda) + \alpha_{t+1}
\left[ 
\varphi(\lambda_{t+1}) + 
			\langle	\nabla \varphi (\lambda_{t+1}), \lambda - \lambda_{t+1}\rangle 
			\right] 
\right\} \\
\nonumber &\geq 
\min_{\lambda}
\left \{
\psi_t(z_t) + \frac{L}{2} \|\lambda - z_t\|_q^2 +  \alpha_{t+1}
\left [ 
\varphi(\lambda_{t+1}) + 
		\langle	\nabla \varphi_\mu(\lambda_{t+1}), \lambda - \lambda_{t+1}\rangle 
			\right ] 
\right \}  \\
\nonumber & \geq
\min_{\lambda}
\left \{
A_t \varphi(y_t) + \frac{L}{2} \| \lambda - z_t \|_q^2 +
\alpha_{t+1}
\left [ 
\varphi_\mu(\lambda_{t+1}) + 
		\langle	\nabla \varphi_\mu(\lambda_{t+1}), \lambda - \lambda_{t+1}\rangle 
			\right ] 
\right \} 
\\
\label{psi_ineq} 
& \geq
\min_{\lambda}
\left \{
A_t \left ( \varphi_\mu(\lambda_{t+1}) + \langle \nabla \varphi_\mu(\lambda_{t+1}), y_t - \lambda_{t+1}\rangle  \right ) +
\frac{L}{2} \|\lambda - z_t\|_q^2 +
\alpha_{t+1}
\left [ \dotsi\right ] 
\right \}.
\end{align}
Step $\lambda_{t + 1} = \tau_t z_t + (1 - \tau_t) y_t$ one can rewrite as the relation $A_{t+1} \lambda_{t+1} = \alpha_{t+1} z_t + A_t y_t$. Using it, we transform:
\begin{align*}
A_t  &\langle  \nabla \varphi_\mu(\lambda_{t+1}), y_t - \lambda_{t+1} \rangle  + \alpha_{t+1} \langle \nabla \varphi_\mu(\lambda_{t+1}), \lambda - \lambda_{t+1} \rangle   =\\ 
& = - A_{t+1} \langle \nabla \varphi_\mu(\lambda_{t+1}), \lambda_{t+1} \rangle  + \alpha_{t+1} \langle \nabla \varphi_\mu(\lambda_{t+1}),
\lambda \rangle + A_t \langle \nabla \varphi_\mu(\lambda_{t+1}), y_t \rangle \\
&= \alpha_{t+1} 
\langle
\nabla \varphi_\mu(\lambda_{t+1}), \lambda - z_t
\rangle.
\end{align*}
Hence we have
\begin{align}
\nonumber A_t &\left ( \varphi_\mu(\lambda_{t+1}) + \langle \nabla \varphi_\mu(\lambda_{t+1}),
y_t - \lambda_{t+1} \rangle \right ) + \frac{L}{2} \|\lambda - z_t\|_q^2 \\
&+ \alpha_{t+1}
\left [ 
\varphi_\mu(\lambda_{t+1}) + 
		\langle	\nabla \varphi_\mu(\lambda_{t+1}), \, \lambda - \lambda_{t+1}
		\rangle
			\right ] =
\\
\label{re_equation}
&= A_{t+1} \varphi_\mu(\lambda_{t+1}) + 
\frac{L}{2} \|\lambda - z_t\|_q^2 +
\alpha_{t+1} 
\langle
\nabla \varphi_\mu(\lambda_{t+1}), \,
\lambda - z_t
\rangle.
\end{align}
After substituting the \eqref{re_equation} in the last expression of~\eqref{psi_ineq} one can use an extended version of Fenchel's inequality for conjugate functions~\cite{nesterov2013disser}:
\[
\langle g, s \rangle + \frac{\xi}{2} \|s\|_q^2 \geq - \frac{1}{2 \xi} \|g\|_{q*}^2.
\]
In our case $g = \nabla \varphi (\lambda_{t+1})$,
$s = \lambda - z_t$, $\xi = \frac{L}{\alpha_{t+1}}$.
Hence,
\begin{equation}
\label{eq_lem1_after_fen}
\psi_{t+1}(z_{t+1}) \geq 
A_{t+1} \varphi_\mu(\lambda_{t+1}) - \frac{\alpha_{t+1}^2}{2L}
\|\nabla \varphi_\mu(\lambda_{t+1})\|^2_q.
\end{equation}
Further, by Lipschitz smoothness of $\varphi_\mu$:
\begin{align*}
    \varphi(y_{t+1}) & \leq
\varphi_\mu(\lambda_{t+1}) + \langle \nabla \varphi_\mu(\lambda_{t+1}), y_{t+1} - \lambda_{t+1}\rangle + \frac{L}{2}
\|y_{t+1} - \lambda_{t+1} \|_q^2 = \\ &
=\min_{\lambda}
\left \{ \varphi_\mu(\lambda_{t+1}) + \langle \nabla \varphi_\mu(\lambda_{t+1}), \lambda - \lambda_{t+1}\rangle  + \frac{L}{2} \|\lambda - \lambda_{t+1}\|_q^2 
			\right
			\} \\ 
		&= \varphi_\mu(\lambda_{t+1}) - \frac{1}{2 L} \| \nabla \varphi_\mu(\lambda_{t+1}) \|_q^2.		
\end{align*}
After multiplying both sides of the resulting inequality by $A_{t+1}$:
\begin{align}
A_{t+1} \varphi_\mu(y_{t+1}) &\leq A_{t+1} \varphi_\mu(\lambda_{t+1}) - \frac{A_{t+1}}{2 L} \| \nabla \varphi_\mu(\lambda_{t+1})\|_q^2 \\
\label{eq_aphi_alam}
&\leq
A_{t+1} \varphi_\mu(\lambda_{t+1}) - \frac{\alpha_{t+1}^2}{2L}
\|\nabla \varphi_\mu(\lambda_{t+1})\|^2_q.
\end{align}
Therefore, due to~\eqref{eq_lem1_after_fen} and \eqref{eq_aphi_alam} we have
$A_{t+1} \varphi_\mu(y_{t+1}) \leq \psi_{t+1}(z_{t+1})$.
\end{proof}
	
\begin{lemma}\label{fgm_lemma}
It can be assumed that $\|\lambda_*\|_q, \|\lambda_0\|_q \leq R_q$. Then the point $\hat{x}_N$ obtained after $N$ iterations of the Algorithm~\ref{alg_fgm} satisfies the condition
\begin{equation}
\varphi_\mu(y_N) - U_\mu(\hat{x}_N)
+ 5 R_q \|(C \hat{x}_N - b)_+\|_q \leq \frac{26 L R_q^2}{A_N}.
\end{equation}
\end{lemma}
\begin{proof}
The vector of dual factors can be localized using the Slater condition. Considering $\|\lambda_t-\lambda_*\|_q \leq \|\lambda_* - \lambda_0\|_q$ (it proved in \cite{ivanova2021numerical}) we have $\|\lambda_t\|_q \leq \|\lambda_t-\lambda_*\|_q + \|\lambda_* - \lambda_0\|_q + \|\lambda_0\|_q \leq 5 R_q$.

From Lemma~\ref{lemma_ineq} and considering also $\|\lambda - \lambda_0\|_q^2 \leq
2 \|\lambda\|_q^2 + 2 \|\lambda_0\|_q^2 \leq
2R_q^2 + 50R_q^2$ we have:
\begin{align*}
A_N \varphi(y_N) & \leq \min_{\lambda} 
\left\{
\frac{L}{2} \|\lambda - \lambda_0\|_q^2 +
\sum_{t = 0}^N \alpha_t \left [ 
\varphi_\mu(\lambda_t) + 
			\langle 
			\nabla \varphi_\mu (\lambda_t), \lambda - \lambda_t \rangle
			\right ] 
\right\}  \\ 
&\leq
 \min_{\|\lambda\|_q \leq 5R_q}
\left \{
\sum_{t = 0}^N \alpha_t \left [ 
\varphi_\mu(\lambda_t) + \langle 
			\nabla \varphi_\mu (\lambda_t), \lambda - \lambda_t \rangle
			\right ] 
\right \} + 26 L R_q^2.
\end{align*}
Substituting expressions for $\varphi_\mu(\lambda_t)$ and $\nabla \varphi_\mu(\lambda_t)$:
\begin{align*}
\sum_{t = 0}^N \alpha_t &\left[ 
\varphi_\mu(\lambda_t) + 
			\langle 
			\nabla \varphi_\mu(\lambda_t), \lambda - \lambda_t \rangle
			\right] \\
			&= \sum_{t = 0}^N \alpha_t
\left[
\langle \lambda_t, b \rangle + U_\mu(x(\lambda_t)) - \langle \lambda_t, C x(\lambda_t) \rangle  + \langle b - C x(\lambda_t), \lambda - \lambda_t \rangle
\right] \\ 
& =   
\sum_{t = 0}^N \alpha_t
\left[
U_\mu(x(\lambda_t)) +
 \langle \lambda, b - C x(\lambda_t) \rangle
\right] \leq A_N \left[
U(\hat{x}_N) + \langle \lambda, b - C \hat{x}_N \rangle
\right], 
\end{align*}
Which finally leads to
\begin{align*}
A_N \varphi_\mu(y_N) & \leq A_N U_\mu(\hat{x}_N) + 26 L R_q^2 + A_N \min_{\|\lambda\|_q \leq 5R_q} \langle \lambda, b - C \hat{x}_N \rangle \\
& = A_N U_\mu(\hat{x}_N) + 26 L R_q^2 - 5 R_q A_N \|(C \hat{x}_N - b)_+\|_q.
\end{align*}

\end{proof}

\begin{theorem}
The point $\hat{x}_N$ ($\hat{x}_{N_\mu}$) obtained after $N$ ($N_\mu$) iterations of the Algorithm \ref{alg_fgm} satisfies the conditions
\begin{equation*}
\max_x U(x) - U(\hat{x}_N) \leq \varepsilon,\quad
  \|(C \hat{x}_N - b)_+\|_q \leq \frac{\varepsilon}{4R_q},
\end{equation*}
if the number of iterations satisfies the following inequality (the first for not strongly concave $u_i$, the second for $\mu$-strongly concave $u_i$):
\begin{equation}
N \geq \left \lfloor 8\sqrt{13} R_q R_p \cdot \frac{\|C\|_{p,q}}{\varepsilon} \right \rfloor,\quad
N_\mu \geq \left \lfloor 2\sqrt{26} \sqrt{R_q R_p} \cdot \frac{\|C\|_{p,q}}{\sqrt{\mu \varepsilon}} \right \rfloor    
\end{equation}
\end{theorem}

\begin{proof}
Due to weak duality, it holds that $\min_\lambda \varphi_\mu(\lambda) \geq \max_x U_\mu(x)$. Due to the fact that $\varphi_\mu(y_N) \geq \min_\lambda \varphi_\mu(\lambda) \geq \max_x U_\mu(x)$, the following estimate follows directly from Lemma \ref{fgm_lemma}:
\begin{equation}\label{function_error}
    \max_x U_\mu(x) - U_\mu(\hat{x}_N) \leq \frac{26 L R_q^2}{A_N}
\end{equation}
Also, by the properties of duality, we have
\begin{align*}
    \nonumber\max_x U_\mu(x) &\geq U_\mu(x) - \langle \lambda_*,
    (C x - b)_+ \rangle \\
    &\geq U_\mu(x) - R_q \|(C x - b)_+\|_q\quad \forall x \in \mathbb{R}^n_+
\end{align*}
where $\lambda_*$ is a minimum point of $\varphi_\mu$. Hence the following estimate follows:
\begin{align*}
    \varphi_\mu(y_N) -  U_\mu(\hat{x}_N) &= (\varphi_\mu(y_N) - \min_\lambda \varphi_\mu(\lambda)) + (\min_\lambda \varphi_\mu(\lambda) - \max_x U_\mu(x)) \\
    &+ (\max_x U_\mu(x) - U_\mu(\hat{x}_N)) \geq -R_q \|(C \hat{x}_N - b)_+\|_q.
\end{align*}
Using Lemma \ref{fgm_lemma} together with the obtained inequality:
\begin{equation}\label{constr_error}
    R_q \|(C \hat{x}_N - b)_+\|_q \leq \frac{13 L R_q^2}{2 A_N}
\end{equation}
Upper-bounding the right-hand side of \eqref{function_error} by $\frac{\varepsilon}{2}$ (or $\varepsilon$, if the $u_i$ are strongly concave) and taking into account this estimate in \eqref{constr_error}, we obtain the desired conditions from theorems and condition for $N$:
\[
\frac{104 L R_q^2}{(N+1)(N+2)} \leq \frac{\varepsilon}{2}
\]
Substituting $L$ from Lemma \ref{lipschitz} and solving for $N$, we obtain the inequalities from the theorem.
\end{proof}

\begin{lemma}
For a $\mu$-strongly concave function $U$ (possibly $\mu = 0$) in a particular case of $p = q = 2$ we have $N \sim \sqrt{nnz(C)}$, i.e.
\begin{equation*}
N = \min\left\{ \left \lfloor 8\sqrt{13} R_q R_p \cdot \frac{\lambda^{1/2}_{max}(C^\top C)}{\varepsilon} \right \rfloor, \left \lfloor 2\sqrt{26} \sqrt{R_q R_p} \cdot \frac{\lambda^{1/2}_{max}(C^\top C)}{\sqrt{\mu \varepsilon}} \right \rfloor \right\},
\end{equation*}
whereas in the case of $p = q = 1$ we have $N \sim \max nnz(C^\top_i)$, and more specifically
\begin{equation*}
N = \min\left\{ \left \lfloor 8\sqrt{13} R_q R_p \cdot \frac{\displaystyle \max_{i=1,...,n}\|C^\top_i\|_1}{\varepsilon} \right \rfloor, \left \lfloor 2\sqrt{26} \sqrt{R_q R_p} \cdot \frac{\displaystyle \max_{i=1,...,n}\|C^\top_i\|_1}{\sqrt{\mu \varepsilon}} \right \rfloor \right\}.
\end{equation*}
\end{lemma}

Let us clarify the question of the choice of norms for the problem under consideration. For the setting under consideration, it turns out to be reasonable to preserve freedom only in the choice of $p \in [1, 2]$, while for the dual problem, due to its structure, as a result of the trade of between a decrease in the Lipschitz constant and an increase in $R$, the most effective choice is $q=2$ \cite{gasnikov2016stoch}. As you can see from the arguments above, another choice of $q$ also does not bring any benefit in term of dependence on the properties of the matrix $C$.

\subsection{Interpretation}
Let us divide the optimizer's responsibility for the primal and dual variables into two natural types of computing agents: connections and vertices. By global process iterations we mean the steps during which each computational agent performs all updates corresponding to the iteration of the Algorithm \ref{alg_fgm} for the corresponding component of the vectors appearing there. Let us show that the described algorithm can be implemented in a real network architecture at reasonable and insignificant costs of communication between computing agents.

So, let us first consider a procedure performed on some connection $j$. In the vector $y_t$ it corresponds to the component $[y_t]_j$, updated based on its corresponding previous value $[\lambda_{t-1}]_j$ and some information about $x(\lambda_t)$. We assume that at the iteration $t$ real data transmission rate $x_i$ for each of the vertices $i$ coincides with $x_i(\lambda_t)$ (the procedures described in this section guarantee this). Note now that when multiplying $x$ by $C_j$ the value of the product will be affected only by those $x_i$ for which $C_{j i} = 1$, that is, the vertices directly in relation to the connection $j$. It is natural then to assume that before updating $[y_t]_j$ the connection $j$ polls all the vertices in relation to it for the values of $x_i$. $[z_t]_j$ and $[\lambda_{t+1}]_j$ components are updated in the same way. Since all connections perform these procedures in parallel, we can assume that in time independent of the dimension of the problem, all components of the vectors $y_t$, $z_t$ and $\lambda_{t+1}$ will be calculated (they will be stored, of course, also distributed).

Now let us move on to considering the procedure performed by the vertex $i$. Its goal is to calculate the optimal data transmission rate $[\hat{x}_{t+1}]_i$. Note that this value can be obtained, if $x_i(\lambda_{t+1})$ is known, by a simple update in constant time and without additional communications. At the same time, it is easy to see that to solve the auxiliary problem and obtain $x_i(\lambda_{t+1})$ one should know only those components $[\lambda_{t+1}]_j$ for which $C_{j i} = 1$ (due to multiplication by $C_i^\top$). This means that a vertex, before executing the procedure, needs to poll all connections that are in relation to it for the values of $[\lambda_{t+1}]_j$. On the other hand, here it is computationally possible to manage such communication protocol, wherein each connection, calculating a value of $[\lambda_{t+1}]_j$, notifies all related nodes about it, while each node collects incoming messages until it receives up-to-date information from all related connection, and after performing the procedure in a symmetric manner will notify all associated connections about new values of $x_i(\lambda_{t+1})$. 

\section{Randomized mirror descent} \label{section_mirror}
\subsection{Theoretical guarantees}

    \begin{algorithm}[H]
		\caption{Stochastic Mirror Descent}\label{alg_mirror}
		\begin{algorithmic}[1]
			
			\Require $x_0$.
			
			\State $I = \varnothing, J = \varnothing$

			\For{$t=0,\, 1, \, \ldots, \, N-1$}
                \If{$C x_t - b \leq \varepsilon$}
    			
    			\State $i \sim \mathcal{U}\{1, ..., n\}$
    			\State $[x_{t+1}]_i = \left[ 
    			    [x_t]_i - \frac{ \varepsilon n}{M_U^2} \nabla u_i([x_t]_i)
                \right]_+$
                \State $I = I \cup \{t+1\}$
                
                \Else
                
                \State $j_t = \arg \max_{j=1,...,m} C_j x_t - b_j$
                \State $i \sim \mathcal{U} \{i: C_{j_t i} = 1\}$
    			\State $[x_{t+1}]_i = \left[ 
    			    [x_t]_i - \frac{ \varepsilon n}{\max_{j=1,...,m}\|C_j\|_{p*}^2} \right]_+$
                \State $J = J \cup \{t+1\}$
                
                \EndIf
            \EndFor
			\State\Return $\hat{x}_N = \frac{1}{|I|} \sum_{t \in I} x_t$
		\end{algorithmic}
	\end{algorithm}
	
In Algorithm \ref{alg_mirror} the stochastic oracles of the gradient of the function $U$ and the gradient of constraints $C x - b$ are used to perform a step of the method, namely, randomized along the vertex $i$, corresponding to the selected term $u_i(x_t)$ and the component of the updated vector $x_t$:
\[
\mathbb{E}_i[e_i \cdot n \nabla u_i([x_t]_i)] = \nabla U(x_t),\qquad
\mathbb{E}_i[e_i \cdot n C_{j_t}] = \nabla (C_{j_t} x_t - b_{j_t}).
\]

We assume that the randomized gradient of $U$ used is bounded:
\begin{equation}\label{u_assump}
|\nabla u_i(x)| \leq M_U.
\end{equation}

Below is a direct consequence of the result on the convergence of the method obtained in \cite{tiapkin2021parallel}.

\begin{theorem}{Theorem 2 \cite{tiapkin2021parallel}}
Point $\hat{x}_N$, obtained after $N$ iterations of the Algorithm \ref{alg_mirror}, and $\hat{\lambda}$, chosen so that $[\hat{\lambda}]_j = \frac{1}{|I|} \frac{M_U^2}{\max_{j=1,...,m}\|C_j\|^2_{p*}} \sum_{t \in J} \mathbbm{1}[j = j_t]$, satisfy the conditions
\[
\mathbb{E}[\max_x U(x) - U(\hat{x})] \leq \mathbb{E}[\varphi(\hat{\lambda}) - U(\hat{x})] \leq \varepsilon,\quad
C x_t - b \leq \varepsilon,
\]
if the number of iterations satisfies the following inequality:
\[
N \geq \left \lceil \frac{72 \max\{M_U, \max_{j=1,...,m}\|C_j\|_{p*}\}^2 n^2 R_p^2}{\varepsilon^2} \right \rceil
\]
\end{theorem}

Note that when choosing the norm $p=1$, we have $\|C_j\|_{p*} = \max_{i=1,...,n} C_{ji} \leq 1$, and therefore the estimate ceases depend on the characteristics of the matrix $C$. 

One of the few problems with the presented approach is that the specific functions $u$ used may not satisfy the assumption \eqref{u_assump}. Note, however, that this problem can be solved~--- for this one can use adaptive versions of mirror descent, similar, for example, to those described in \cite{bayandinaadaptive}. Thus, it is possible to obtain estimates similar to those presented above, but including, instead of $M_U$, the constants of the form $(\frac{1}{N} \sum_{t=1}^N M_t^2)^{1/2}$, where $M_t$ is the adaptively selected constant value at iteration $t$.

\subsection{Interpretation}

Let us analyze, similarly to the previous considered method, the computational and communication protocol corresponding to the Algorithm \ref{alg_mirror}. In this case, direct duality arises in the described approach not constructively, but only theoretically. This means that in order to calculate the prices of information transmission for each of the connections, it is not necessary to know the values $x_i$. At the same time, for calculating prices, it becomes necessary to know some data about the structure of the network and the problem, as well as additional information that appears in the course of global iterations of the method. In addition, the very structure of descent with switches requires verification of the fulfillment of all constraints of the problem, which requires information about the operation of all connections (or about the values of the data transmission rates for all vertices, the decision on the preferred verification method is made only for reasons of communication complexity). This gives rise to the need to introduce some kind of decision-making center that aggregates information about the operation of the network and provides it upon request to individual computing agents.

Let's describe the procedures performed by the decision center. At the beginning of the iteration, the center polls the connections or vertices for the fulfillment of the problem's constraints, after which it decides on the type of iteration ($t \in I$ or $t \in J$). If it turns out that some of the connections are overloaded, then the most overloaded of them is selected, which has the number $j_t$. This completes the center procedure, it is repeated only after one of the vertices notifies it of the completed update of its component $[x_{t+1}]_i$: along with this, the center can also notify with the values $|I|$ and $\{j_t\}_{t=0}^N$ all the connections in the network to allow them to calculate their own prices~--- this, however, is not necessary when the goal is only to find the optimal data transmission rates.

Now let's move on to the procedure performed on the vertex $i$. At the beginning of a local iteration (which is automatically launched every fixed period of time), the vertex appeals the center to obtain information about the iteration type and value $j_t$. If the vertex $i$ and the connection $j_t$ are in a relation or the iteration type matches $t \in I$, the vertex accordingly updates the value of the data transmission rate $[x_{t+1}]_i$ and notifies the center (only about the fact of the update, transfer of the new value is not required). If we assume that the number of the vertex $i$ that first requested information from the center is uniform distributed, the resulting protocol will is correspond to the analyzed algorithm.

Note also that the center does not need to poll all connections for overloading every time, since when only one component $[x_{t+1}]_i$ is updated, the vertex can inform the center not only about the fact of overloading, but also about the new value of this components, and then the center can update the values of the constraints more efficiently, without having to go to all the connections of the network every time.

\section{Numerical experiments} \label{section_num}
Let's move on to the description of numerical experiments. The results presented below were obtained on a macOS 10.15.7 PC running a 3 GHz 6-core Intel Core i5 processor, with a Python 3.8.2 interpreter in a Jupyter Notebook environment, calculations were performed on variables of numpy.float64 type. The source code for setting up experiments is available at  \url{https://github.com/dmivilensky/Network-utility-maximization-2021}.

\begin{figure}[H]
	\centering
	\vspace{-0.6cm}
	\includegraphics[width=0.99\textwidth]{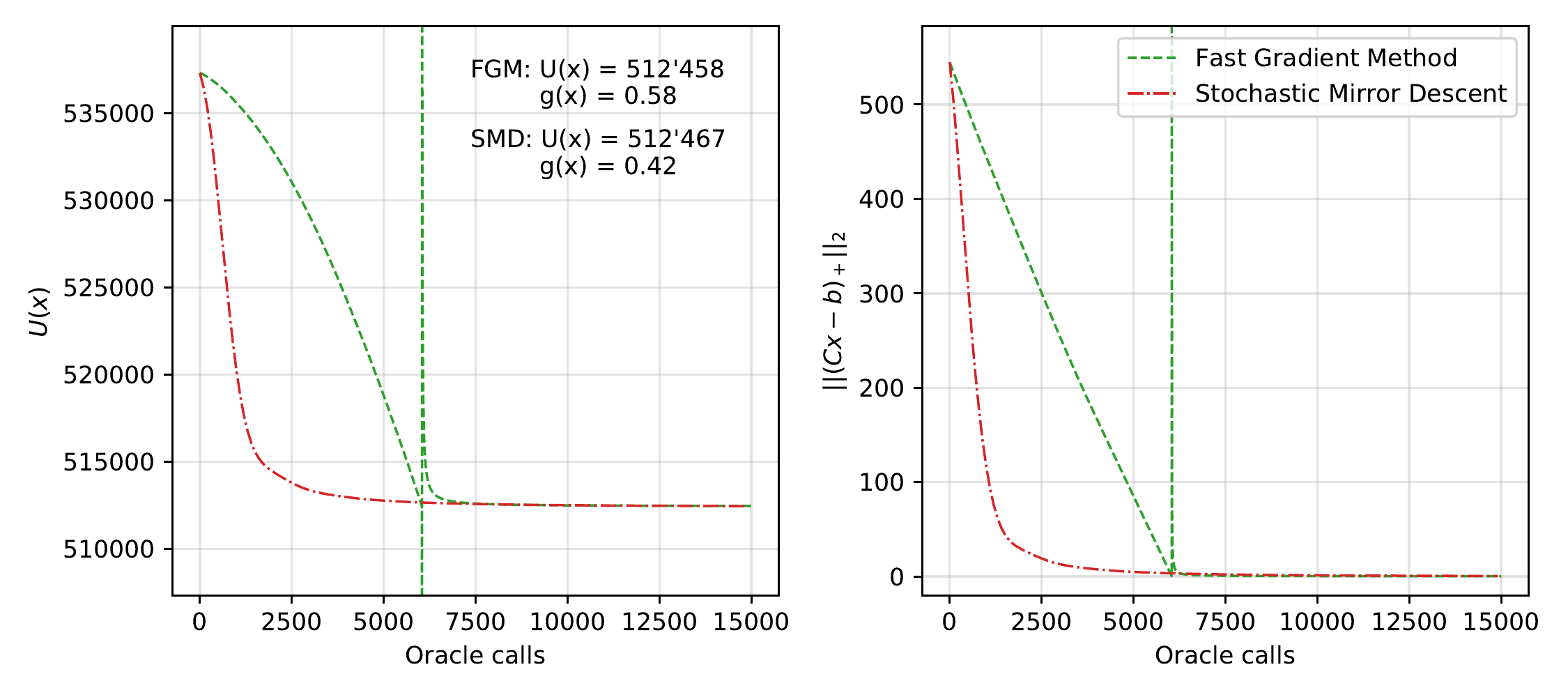}
	\vspace{-0.4cm}
	\caption{Algorithm \ref{alg_fgm} and Algorithm \ref{alg_mirror} comparison}
	\label{fig:comparison}
\end{figure}

Figure \ref{fig:comparison} shows the dynamics of the values of the total utility function and the constraint residual for the sequences of data transmission rates generated by Algorithm \ref{alg_fgm} and Algorithm \ref{alg_mirror} in the course of solving problem \eqref{primal}. For this experiment, we consider a particular formulation of the problem with $p=q=2$, $m=40, n=100$, matrix $C$ with i.i.d. drawn components, such that $nnz(C) \approx 0.001 \cdot m \cdot n$, vector $b$ with i.i.d. components drawn from uniform distribution $\mathcal{U}[0, n]$. We consider utility functions $u_i$ of the form
\[
u_i(x) = a_i \cdot x - \frac{\sigma \cdot n}{2} \cdot x^2,
\]
where $a$ is vector with i.i.d. components drawn from uniform distribution $\mathcal{U}[1, 50]$, $\sigma = 0.001$. Hence, closed form solution of \eqref{greedy} is
\[
x_i(\lambda) = \frac{\left(a - C^\top \cdot \lambda\right)_+}{n \cdot \sigma}.
\]

As we can see from the presented figure, there is a distinct moment of the first switching of Algorithm~\ref{alg_mirror} from ``constraint correction mode'' to the ``combined mode', in which the utility function is also optimized. The abscissa axis of the graphs measures the number of one component oracle calls. In this experiment, it can be seen that Algorithm~\ref{alg_mirror} obtains a slightly better solution than Algorithm~\ref{alg_fgm}, using the same number of calls to individual components. The real working time of Algorithm~\ref{alg_mirror} in the simulation exceeds that for Algorithm~\ref{alg_fgm}, however, in a real-life environment of a computer network, the main contribution during operation is made by delays in waiting for responses from network participants, which is more objectively reflected by the number of oracle calls. 

Optimizing the constraints of the problem at the first stage of the operation of Algorithm~\ref{alg_mirror} can take a long time, determined only by the properties of the problem (that is, arbitrarily long). At the same time, if there is a point in the feasible set, for example, describing the present empirical distribution of data transmission rates in the network, Algorithm~\ref{alg_mirror} immediately starts in main mode. In any case, both proposed methods turn out to be competitive alternatives.

\section*{Conclusion}

Following the approach of updating individual data transmission rates and pricing network connections, we propose two methods to maximize the utility of the network. The first is obtained as a generalization of the fast gradient method to the case of not strongly concave utility functions, using the dual smoothing technique. The second is an adaptation of switched stochastic mirror descent for constrained problems. For both proposed methods, a detailed description of the protocols of their operation in the environment of a real-life computer network was presented, taking into account distributed data storage and limited communication capabilities. Numerical experiments on synthetic architectures of computer networks allow us to compare the practical efficiency of the proposed algorithms.

\bibliographystyle{splncs04}
\bibliography{main}
\end{document}